\documentclass[11pt,reqno]{amsart}
\usepackage[T1]{fontenc}
\usepackage{graphicx}
\usepackage{setspace}

\usepackage{color}
\definecolor{MyLinkColor}{rgb}{0,0,0.4}

\newcommand{\R}{{\mathbb R}}

\newcommand{\Z}{{\mathbb Z}}

\newcommand{\N}{{\mathbb N}}

\newcommand{\s}{\mathbb{S}}

\newcommand{\cF}{\mathcal{F}}

\newcommand{\kL}{\mathcal{L}}
\newcommand{\cC}{\mathcal{C}}

\newcommand{\la}{\langle}
\newcommand{\ra}{\rangle}

\newcommand{\wt}{\widetilde}

\newcommand{\ov}{\overline}

\newcommand{\un}{\underline}

\newcommand{\p}{\partial}

\newcommand{\e}{\varepsilon}

\newcommand{\0}{\Omega}
\newcommand{\G}{\Gamma}

\newcommand{\tr}{\mathop{\rm tr}\nolimits}

\newcommand{\spa}{\mathop{\rm span}\nolimits}
\newcommand{\im}{\mathop{\rm Im}\nolimits}
\newcommand{\ke}{\mathop{\rm Ker}\nolimits}

\newtheorem{thm}{Theorem}[section]
\newtheorem{prop}[thm]{Proposition}
\newtheorem{defn}[thm]{Definition}
\newtheorem{lemma}[thm]{Lemma}

\theoremstyle{remark} 
\newtheorem{rem}[thm]{Remark}

\numberwithin{equation}{section}   

\title[Steady water waves with unbounded vorticity]{Steady periodic water waves with unbounded  vorticity: equivalent formulations and existence results}

\author[C. I. Martin]{Calin Iulian Martin}
\address{Institut  f\" ur Mathematik, Universit\" at Wien, Nordbergstra{\ss}e 15,
1090 Wien, Austria}
\email{calin.martin@univie.ac.at}

\author[B.--V. Matioc]{Bogdan--Vasile Matioc}
\address{Institut f\" ur Mathematik, Universit\" at Wien, Nordbergstra{\ss}e 15,
1090 Wien, Austria}
\email{bogdan-vasile.matioc@univie.ac.at}


\usepackage[colorlinks=true,linkcolor=MyLinkColor,citecolor=MyLinkColor]{hyperref} 

\begin{document}

\begin{abstract}
In this paper we consider the steady water wave problem for waves that possess a merely  $L_r-$integrable   vorticity, with $r\in(1,\infty)$ being arbitrary.   
We first establish  the equivalence of the three formulations--the velocity formulation, the stream function formulation, and the height function formulation--
in the setting of strong solutions, regardless of the value of $r$. 
Based upon this result and using a suitable  notion of  weak solution for the height function formulation, we then establish, by means of local bifurcation theory, the existence of small amplitude capillary and capillary-gravity
water waves with a  $L_r-$integrable vorticity.
\end{abstract}

\maketitle

\section{Introduction}\label{Sec:1}

We consider the classical problem of traveling waves that propagate at the surface of a two-dimensional
inviscid and  incompressible fluid of finite depth.
Our setting is general enough to incorporate the case when the vorticity of the  fluid is merely $L_r-$integrable, with   $r>1$ being arbitrary.
The existence of solutions of the Euler equations  in   $\R^n$ describing flows with an unbounded vorticity distribution has been 
addressed lately by several authors, cf. \cite{ Ke11, MB02, Vi00} and the references therein,  
whereas  for traveling free surface  waves  in two-dimensions  there are so far no
existence results which allow for a merely $L_r$-integrable vorticity.   

In our setting, the hydrodynamical problem is modeled by the steady Euler equations,  to which we refer to as the {\em velocity formulation}. 
For classical solutions and in the absence of stagnation points, 
there are two equivalent formulations, namely the {\em stream function} and the {\em height function} formulation, the latter being related to the semi-Lagrangian  Dubreil-Jacotin transformation. 
This equivalence property stays at the basis of the existence results of classical solutions with general H\"older continuous vorticity, cf. \cite{CoSt04} for gravity waves and \cite{W06b, W06a} for waves with capillarity.
Very recent, taking advantage of the  weak 
formulation of the governing equations, it was rigorously established that there exist  gravity  waves \cite{CS11} and capillary-gravity waves \cite{CM13xx, MM13x} with a discontinuous and bounded vorticity.
The waves found  in the  latter references are obtained in the setting of strong solutions when the equations of motion are satisfied in $L_r,$ $r>2$ in  \cite{CS11}, respectively $L_\infty$ in \cite{CM13xx, MM13x}. 
The authors of \cite{CS11} also prove  the  equivalence of the formulations in the setting of $L_r$-solutions, under the restriction that $r>2$.
Our first main result, Theorem \ref{T:EQ}, establishes the equivalence of the three formulations for strong solutions that possess Sobolev  and weak H\"older regularity.
For this we rely on the regularity properties of such solutions, cf. \cite{AC11, EM13x, MM13x}.
This equivalence holds for gravity, capillary-gravity, and pure capillary waves without stagnation points and with a $L_r$-integrable vorticity, without making any restrictions on $r\in(1,\infty).$ 

The equivalence result Theorem \ref{T:EQ} stays at the basis of our second main result, Theorem \ref{T:MT}, where we establish the existence of small amplitude capillary and capillary-gravity water waves having  a 
$L_r-$integrable vorticity distribution for any $r\in(1,\infty).$
On physical background, studying waves with an unbounded vorticity is relevant in the setting of small-amplitude  wind generated waves, when  
capillarity plays an important role. 
These waves may possess a shear layer of high vorticity adjacent to the wave surface \cite{O82, PB74}, fact which motivates us to consider 
unbounded vorticity distributions.
{Moreover, an unbounded  vorticity at  the bed is also  physically relevant, for example when describing turbulent flows along smooth channels (see the empirical law on page 109 in \cite{B62}).}

In contrast to the  irrotational case when, in the absence of an underlying current,  the qualitative features of the flow are well-understood 
  \cite{C12, Co06, DH07}, 
in the presence of a underlying, even uniform \cite{CoSt10, HH12, SP88},  current  many aspects of the flow are more difficult to study, or are even untraceable, and one has to rely often 
on numerical simulations, cf. \cite{KO1, KO2, SP88}.
   For example, by allowing for a discontinuous vorticity, the latter studies display the influence of a favorable or adverse wind on 
   the amplitude of the waves, or describe    extremely high rotational waves
and the the flow pattern  of waves with eddies. 

The rigorous existence of waves with capillarity was obtained first in the setting of irrotational waves \cite{MJ89, JT85, JT86,  RS81} and it was only recently 
extended to  the setting of waves with constant vorticity and stagnation points \cite{ CM13a,  CM13, CM13x}   (see also \cite{CV11}).           
 In the context of waves with a general H\"older continuous \cite{W06b, W06a} or discontinuous \cite{CM13xx, MM13x} vorticity 
the existence results are obtained by using the height function formulation and  concern only small amplitude
waves without stagnation points. 
 Theorem \ref{T:MT}, which is the first rigorous existence result for waves with unbounded vorticity, is obtained by taking advantage of the 
weak interpretation  of the height function formulation.
 More precisely, recasting the nonlinear second-order boundary condition on the surface into a nonlocal  and nonlinear equation of order zero 
  enables us to introduce the
 notion of weak  (which is  shown to be strong) solution for the problem in a suitable  analytic setting.
By means of local bifurcation theory and ODE techniques  we then find local real-analytic curves consisting, with the exception of a single  laminar flow solution, only of non-flat 
symmetric capillary (or capillary-gravity) water waves.   
The methods we apply are facilitated by the presence of capillary effects (see e.g. the proof of Lemma \ref{L:4}), though not on the value of the surface tension coefficient,  
the existence question for pure gravity waves with  unbounded vorticity being left as an open
problem.

 The outline of the paper is as follows: we  present in Section \ref{Sec:2}  the mathematical setting and establish the equivalence of the formulations  in Theorems \ref{T:EQ}. 
We end the section by stating our main existence result Theorem \ref{T:MT}, the Section \ref{Sec:3} being dedicated to its proof.

\section{Classical formulations of  the steady water wave problem and the main results}\label{Sec:2}

Following a steady periodic wave from a reference frame which moves in the same direction as the wave and with the same speed $c$,
 the equations of motion  are    the steady-state Euler equations
\begin{subequations}\label{eq:P}
  \begin{equation}\label{eq:Euler}
\left\{
\begin{array}{rllll}
({u}-c) { u}_x+{ v}{ u}_y&=&-{ P}_x,\\
({ u}-c) { v}_x+{ v}{ v}_y&=&-{ P}_y-g,\\                                 
{ u}_x+{v}_y&=&0
\end{array}
\right.\qquad \text{in $\0_\eta,$}
\end{equation}
with $x$ denoting the direction of wave propagation and $y$ being the height coordinate.
We  assumed that the free surface of the wave is the graph $y=\eta(x),$ that the fluid has constant unitary density,  and that the flat fluid bed is located at $y=-d$. 
Hereby, $\eta$ has zero integral mean over a period and $d>0$ is the average mean depth of the fluid.   
Moreover, $\0_\eta $ is the two-dimensional fluid domain   
\[
\0_\eta:=\{(x,y)\,:\,\text{$ x\in\s $ and $-d<y<\eta(x)$}\},
\]
with  $\s:=\R/(2\pi\Z)$ denoting the unit circle.
This notation expresses the $2\pi$-periodicity  in   $x $ of $\eta,$ of the velocity field $(u, v), $ and of the pressure $P$.  
  The equations  \eqref{eq:Euler} are supplemented by the following boundary conditions
\begin{equation}\label{eq:BC}
\left\{
\begin{array}{rllll}
 P&=&{P}_0-\sigma\eta''/(1+\eta'^2)^{3/2}&\text{on $ y=\eta(x)$},\\
 v&=&({ u}-c) \eta'&\text{on $ y=\eta(x)$},\\
 v&=&0 &\text{on $ y=-d$},
\end{array}
\right.
\end{equation}
the first relation being a consequence of Laplace-Young's equation which  states that the pressure jump across an interface is proportional to the mean curvature of the  interface.
We used   $ P_0$ to denote  the constant atmospheric pressure and $\sigma>0$  is the surface tension coefficient. 
Finally, the  vorticity of the flow is the scalar function
\begin{equation*}
\omega:= { u}_y-{ v}_x\qquad\text{in $\0_\eta$.}  
\end{equation*}
\end{subequations}

The velocity formulation \eqref{eq:P} can be re-expressed in terms of the  stream function $\psi $, which is introduced via the relation $\nabla \psi=(-v,u-c)$ in $\0_\eta$, cf. Theorem \ref{T:EQ}, and it becomes a   free boundary problem
\begin{equation}\label{eq:psi}
\left\{
\begin{array}{rllll}
\Delta \psi&=&\gamma(-\psi)&\text{in}&\0_\eta,\\
\displaystyle|\nabla\psi|^2+2g(y+d)-2\sigma\frac{\eta''}{(1+\eta'^2)^{3/2}}&=&Q&\text{on} &y=\eta(x),\\
\psi&=&0&\text{on}&y=\eta(x),\\
\psi&=&-p_0&\text{on} &y=-d.
\end{array}
\right.
\end{equation}
Hereby, the constant $p_0<0$ represents the relative mass flux,  $Q\in\R$ is related to the total head,
and the function $\gamma:(p_0,0)\to\R$ is the vorticity function, that is 
\begin{equation}\label{vor}
\omega(x,y)=\gamma(-\psi(x,y)) 
\end{equation}
for  $(x,y)\in \0_\eta.$
The equivalence of the velocity formulation \eqref{eq:P} and of the stream function formulation \eqref{eq:psi} in the setting of classical solutions without stagnation points, that is when
 \begin{equation}\label{SC}
u-c<0\qquad\text{in $\ov \0_\eta$}
\end{equation}
has been established   in \cite{Con11, CoSt04}.
We emphasize that the assumption \eqref{SC} is crucial when proving the existence of the vorticity function $\gamma$.
Additionally,  the condition \eqref{SC} guarantees in the classical setting considered in these references that the semi-hodograph transformation
 $\Phi:\ov\0_\eta\to\ov\0$  given by
\begin{equation}\label{semH}
\Phi(x,y):=(q,p)(x,y):=(x,-\psi(x,y))\qquad \text{for $(x,y)\in\ov\0_\eta$},
\end{equation}
where $\0:=\s\times(p_0,0),$ 
is a diffeomorphism. 
This property is   used to show that the  previous two formulations \eqref{eq:P} and \eqref{eq:psi} can be re-expressed in terms of the so-called
  height function   $h:\ov \0\to\R$ defined by  
\begin{equation}\label{hodo}
h(q,p):=y+d \qquad\text{for $(q,p)\in\ov\0$}.
\end{equation}                                                 
 More precisely, one obtains a quasilinear elliptic  boundary value problem  
\begin{equation}\label{PB}
\left\{
\begin{array}{rllll}
(1+h_q^2)h_{pp}-2h_ph_qh_{pq}+h_p^2h_{qq}-\gamma h_p^3&=&0&\text{in $\0$},\\
\displaystyle 1+h_q^2+(2gh-Q)h_p^2-2\sigma \frac{h_p^2h_{qq}}{(1+h_q^2)^{3/2}}&=&0&\text{on $p=0$},\\
h&=&0&\text{on $ p=p_0,$}
\end{array}
\right.
\end{equation}
the condition \eqref{SC} being re-expressed as
\begin{equation}\label{PBC}
 \min_{\ov \0}h_p>0.
\end{equation}

The equivalence of the three formulations \eqref{eq:P}, \eqref{eq:psi}, and \eqref{PB} of the water wave problem, when the vorticity is only $ L_r-$integrable, has not been established yet for the full range $r\in(1,\infty)$.
In the context of strong solutions, when the equations of motion are assumed to hold in $L_r$, there is a recent result \cite[Theorem 2]{CS11}  established in the absence of capillary forces. 
This result though is restricted to the  case when $r>2,$ this condition being related to  Sobolev's embedding $W^2_r\hookrightarrow C^{1+\alpha}$ in two dimensions.  
In the same context, but for  solutions that possess weak H\"older regularity, there is a further equivalence result \cite[Theorem 1]{VZ12}, but again one has  to restrict the  range of H\"older exponents.
Our equivalence result, cf. Theorem \ref{T:EQ}  and Remark \ref{R:-2} below, is true for all $r\in(1,\infty)$ and was obtained in the setting of strong solutions that possess, additionally to Sobolev regularity,   weak H\"older regularity, the 
 H\"older exponent being related in our context to Sobolev's embedding   in only one dimension.
 This  result enables us to establish, cf. Theorem \ref{T:MT} and Remark \ref{R:0},  the existence of small-amplitude capillary-gravity and pure capillary  water waves with $L_r-$integrable vorticity function for any $r\in(1,\infty).$

We   denote in the following by $\tr_0$ the trace operator with respect to the boundary component $p=0$ of $\ov\0,$ that is $\tr_0v=v(\cdot,0)$ for all $v\in C(\ov\0).$
In the following, we use several times  the following product formula 
\begin{equation}\label{PF}
\qquad  \p(uv)=u\p v+v\p u\qquad\text{for all $u,v\in W^1_{1,loc}$ with $uv, u\p v+v\p u\in L_{1,loc},$}
\end{equation}
cf. relation (7.18) in \cite{GT01}.

\begin{thm}[Equivalence of the three formulations]\label{T:EQ}
Let $r\in(1,\infty)$ be given and define  $\alpha=(r-1)/r\in(0,1).$
Then, the following are equivalent
\begin{itemize}
\item[$(i)$] the height function formulation together with  \eqref{PBC} for $h\in C^{1+\alpha}(\ov\0)\cap W^2_r(\0)$ such that $\tr_0 h \in W^2_r(\s)$, and $\gamma\in L_r((p_0,0))$;
\item[$(ii)$] the stream function formulation for $\eta\in W^2_r(\s)$,  $\psi\in C^{1+\alpha}(\ov\0_\eta)\cap W^2_r(\0_\eta)$ satisfying $\psi_y<0$ in $\ov\0_\eta$, and  $\gamma\in L_r((p_0,0))$;
\item[$(iii)$] the velocity formulation together with \eqref{SC} for $u,v, P\in C^{\alpha}(\ov\0_\eta)\cap W^1_r(\0_\eta),$   and $\eta\in W^2_r(\s).$
 \end{itemize}
 \end{thm}

 \begin{rem}\label{R:-2}
  Our equivalence result is true for both  capillary and  capillary-gravity water waves.
  Moreover, it is also true for pure gravity waves when the proof  is similar, with modifications just when proving that $(iii)$ implies $(i)$: 
instead of using  \cite[Theorem 5.1]{MM13x} one has to  rely on the corresponding regularity 
  result  established for gravity waves, cf. Theorem 1.1 in \cite{EM13x}.  
  
  We emphasize also that the condition $\tr_0 h \in W^2_r(\s)$ requested at $(i)$ is not a restriction. 
  In fact, as a consequence of $h\in C^{1+\alpha}(\ov\0)\cap W^2_r(\0)$ being a strong solution of \eqref{PB}-\eqref{PBC} for $\gamma\in L_r((p_0,0))$, we know that the wave surface and  all the other streamlines 
 are real-analytic curves,
  cf. \cite[Theorem 5.1]{MM13x}
  and \cite[Theorem 1.1]{EM13x}. 
  Particularly, $\tr_0 h$ is a real-analytic function, i.e. $\tr_0 h\in C^\omega(\s)$.
  Furthermore,   in view of the same references, all  weak solutions  $h\in C^{1+\alpha}(\ov\0)$ of \eqref{PB}, cf. Definition \ref{D:1} (or \cite{CS11} for gravity waves), satisfy $h \in W^2_r(\s)$.
 \end{rem}

\begin{proof}[Proof of Theorem \ref{T:EQ}] Assume first $(i)$ and let
 \begin{align}\label{DEF}
  d:=\frac{1}{2\pi}\int_0^{2\pi} \tr_0 h\, dq\in(0,\infty)\qquad\text{and}\qquad \eta:=\tr_0 h-d\in W^2_r(\s). 
 \end{align} 
We prove that there exists a unique function $\psi\in C^{1+\alpha}(\ov\0_\eta)$ with the property that 
\begin{align}\label{PP1}
 y+d-h(x,-\psi(x,y))=0\qquad\text{for all $(x,y)\in\ov\0_\eta.$}
\end{align}
 To this end, let $H:\s\times\R\to\R$ to be a continuous extension of $h$ to $\s\times\R$, having the property that $H(q,\cdot)\in C^1(\R)$ is strictly increasing  and has a bounded derivative for all $q\in\s.$ 
Moreover, define the function $F:\s\times\R\times\R\to \R$ by setting
\begin{align*}
 F(x,y,p)=y+d-H(x,p).
\end{align*}
For every fixed $x\in\s$, we have 
\[\text{$F(x,\cdot,\cdot)\in C^1(\R\times\R,\R),$ \quad $F(x,\eta(x),0)=0$,\quad and $F_p(x,\cdot,\cdot)=-H_p(x,\cdot)<0$. }\]
Using the implicit function theorem, we find a  $C^1-$function $\psi(x,\cdot):(\eta(x)-\e,\eta(x)+\e)\to\R$ with the property that
\begin{align*}
 F(x,y,-\psi(x,y))=0 \qquad\text{for all $y\in (\eta(x)-\e,\eta(x)+\e)$}.
\end{align*}
As $\psi_y(x,y)=1/F_p(x,y,-\psi(x,y))$, we deduce that $\psi(x,\cdot) $ is a strictly decreasing function which maps, due to the boundedness of $H_p(x,\cdot)$, bounded intervals onto bounded intervals. 
Therefore, $\psi(x,\cdot) $ can be defined on  $(-\infty,\eta(x)]$.
In view of $F(x,-d,p_0)=0,$ we get that $\psi(x,-d)=-p_0$ for each $x\in\s.$
Observe also that, due to the periodicity of $H$ and $\eta$, $\psi$ is $2\pi-$periodic with respect to $x$, while, because use of  $F\in C^1(\s\times\R\times[p_0,0]),$ we have $\psi\in C^1(\0_\eta)$.
Since the relation \eqref{PP1} is satisfied in $\ov\0_\eta$, it is easy to see now that in fact $\psi\in C^{1+\alpha}(\ov\0_\eta).$

In order to show that $\psi$ is the desired stream function, we prove that $\psi\in W^2_r(\0_\eta).$
Noticing  that the relation \eqref{PP1} yields
\begin{align}\label{RE1}
 \psi_y(x,y)=-\frac{1}{h_p(x,-\psi(x,y))}\qquad\text{and}\qquad \psi_x(x,y)=\frac{h_q(x,-\psi(x,y))}{h_p(x,-\psi(x,y))}
\end{align}
in $\ov\0_\eta,$    the variable transformation \eqref{semH}, integration by parts, and the fact that $h$ is a strong solution of \eqref{PB} yield  
\begin{align*}
\Delta\psi[\wt\phi]=& -\int_{\0_\eta}\left(\psi_y\wt\phi_y+\psi_x\wt\phi_x\right)\, d(x,y)=-\int_{\0}\left(h_q\phi_q-\frac{1+h_q^2}{h_p}\phi_p\right)\, d(q,p)\\
 =&\int_\0\left(h_{qq}-\frac{2h_qh_{pq}}{h_p}+\frac{(1+h_q^2)h_{pp}}{h_p^2}\right)\phi\, d(q,p)=\int_\0(\gamma  \phi)h_p\, d(q,p) \\
=&\int_{\0_\eta} \gamma(-\psi)\wt \phi \, d(x,y)
\end{align*}
 for all $\wt \phi\in C^\infty_0(\0_\eta),$ whereby we set  $\phi:=\wt\phi\circ\Phi^{-1}.$
This shows that $\Delta\psi=\gamma(-\psi)\in L_r(\0_\eta)$. 
Taking into account that $\psi(x,y)=p_0(y-\eta(x))/(d+\eta(x))$ for  $(x,y)\in\p\0_\eta,$ whereby in fact $\eta\in C^\omega(\s),$ c.f. Remark \ref{R:-2},
we find by elliptic regularity, cf. e.g. \cite[Theorems 3.6.3 and 3.6.4]{CW98}, that $\psi\in W^2_r(\0_\eta).$ 
It is also easy to see that $(\eta,\psi)$ satisfy also the second relation of \eqref{eq:psi}, and this completes our arguments in this case.

We now show that  $(ii)$ implies $(iii)$.
To this end, we define
 \begin{align}\label{PP2}
   (u-c,v)&:=(\psi_y,-\psi_x)\qquad\text{and}\qquad P:=-\frac{|\nabla\psi|^2}{2}-g(y+d)-\Gamma(-\psi)+P_0+\frac{Q}{2}, 
  \end{align}
  where $\Gamma$ is given by
  \begin{equation}\label{E:G}\Gamma(p):=\int_0^p\gamma(s)\, ds\qquad\text{for $p\in[p_0,0].$}\end{equation}
 Clearly, we have that  $u,v\in  C^{\alpha}(\ov\0_\eta) \cap W^1_r(\0_\eta)  $ and  $\Gamma\in C^\alpha([p_0,0])\cap W^1_r((p_0,0)).$
Moreover, because $\psi\in  C^{1+ \alpha}(\ov\0_\eta)\cap W^2_r(\0_\eta),$ the  formula \eqref{PF} shows that $|\nabla\psi|^2\in W^1_r(\0_\eta),$
and therefore also $P\in  C^{\alpha}(\ov\0_\eta)\cap W^1_r(\0_\eta).$
The boundary conditions \eqref{eq:BC} are easy to check.
Furthermore, the conservation of mass equation  is a direct consequence of the first relation of \eqref{PP2}. 
We are left with the conservation of momentum equations.
Therefore, we observe   the function  $\Gamma(-\psi)$ is differentiable almost everywhere and its partial  derivatives belong to $L_r(\0_\eta)$, meaning that $\Gamma(-\psi)\in W^1_r(\0_\eta)$, cf. \cite{DD12},
the gradient $\nabla(\Gamma(-\psi))$ being determined by  the chain rule. 
  Taking now the weak derivative with respect to $x$ and $y$ in the second equation of \eqref{PP2}, respectively, we obtain in view of \eqref{PF}, the conservation of momentum equations.
  
  We now  prove that $(iii)$ implies $(ii)$.
  Thus, choose $u,v, P\in C^{\alpha}(\ov\0_\eta)\cap W^1_r(\0_\eta) $  and  $\eta\in W^2_r(\s) $   such that $(\eta, u-c, v,P)$ is a solution of the velocity formulation.
  We define
  \begin{equation}
   \psi(x,y):=-p_0+\int_{-d}^{y} (u(x,s)-c)\, ds\qquad\text{for $(x,y)\in\ov\0_\eta,$}
  \end{equation}
with $p_0$ being a negative constant.
  It is not difficult to see that the function $\psi$ belongs to $ C^{1+\alpha}(\ov\0_\eta)\cap W^2_r(\0_\eta) $  and that it satisfies $\nabla\psi=(-v,u-c).$
The latter relation allows us to pick $p_0$ such that $\psi=0$ on $y=\eta(x).$
Also, we have that $\psi=-p_0$ on the fluid bed.
We next show that the  vorticity of the flow satisfies the relation \eqref{vor} for some $\gamma\in L_r((p_0,0)).$  
To this end, we proceed as in \cite{ BM11} and use the property that the mapping $\Phi$ given by \eqref{semH} is an isomorphism of class $C^{1+\alpha}$ to compute that
\begin{align*}
 \p_q (\omega\circ \Phi^{-1})[ \phi]=&\int_{\0_\eta} (v_x-u_y)((u-c) \wt\phi_x+v\wt\phi_y)\, d(x,y)
\end{align*}
for all $ \phi\in C^\infty_0(\0).$ 
Again, we set $\wt\phi :=\phi\circ\Phi\in C^{1+\alpha}_0(\ov\0_\eta).$
Since our assumption $(iii)$ implies that $(u-c)^2$ and $v^2$ belong to $W^1_r(\0_\eta)$, cf. \eqref{PF},    density arguments, \eqref{eq:Euler},  and integration by parts yield
\begin{align*}
 \p_q (\omega\circ \Phi^{-1})[ \phi]=&\int_{\0_\eta} ((u-c)v_x+vv_y)\wt\phi_x\, d(x,y)-\int_{\0_\eta} ((u-c)u_x+vu_y)\wt\phi_y\, d(x,y)\\[1ex]
=&-\int_{\0_\eta} (P_y+g)\wt\phi_x\, d(x,y)+\int_{\0_\eta} P_x\wt\phi_y\, d(x,y)=0.
\end{align*}
Consequently,  there exists $\gamma\in L_r((p_0,0))$ with the property that $\omega\circ\Phi^{-1}=\gamma$   almost everywhere in $\0$. 
This shows that \eqref{vor} is satisfied in $L_r(\0_\eta).$
Next, we observe that the same arguments used when proving that $(ii)$ implies $(iii)$ yield that the energy 
\[
E:=P+\frac{|\nabla\psi|^2}{2}+g(y+d)+\Gamma(-\psi) 
\]
is constant in $\ov\0_\eta.$
Defining $Q:=2(E-P_0),$ one can now easily see that $(\eta,\psi)$ satisfies \eqref{eq:psi}, and we have established $(ii)$.

In the final part of the proof we assume that $(ii)$ is satisfied and we prove  $(i)$.  
Therefore, we let  $h:\ov\0\to\R$ be the mapping defined by \eqref{hodo} (or equivalently \eqref{PP1}).
Then, we get that $ h\in C^{1+\alpha}(\ov\0)$ verifies the relations \eqref{PP1} and  \eqref{RE1}.
Consequently, $\tr_0 h\in W^2_r(\s)$ and  one can easily see that the boundary conditions of \eqref{PB} and \eqref{PBC} are satisfied.
In order to show that $h$ belongs to $W^2_r(\0)$ and it also solves the first relation of \eqref{PB}, we  observe that the first equation of \eqref{eq:psi} can be written in the equivalent form
\begin{equation}\label{eq:psi2}
(\psi_x\psi_y)_x+\frac{1}{2}\left(\psi_y^2-\psi_x^2\right)_y+(\Gamma(-\psi))_y=0\qquad\text{in $L_r(\0_\eta)$}.
\end{equation}
Therewith and  using the change of variables \eqref{semH}, we find   
\begin{align*}
 &\int_\0\frac{h_q}{h_p} \phi_q-\left(\Gamma+\frac{1+h_q^2}{2h_p^2}\right)\phi_p\,  d(q,p)\\[1ex]
 &=-\int_{\0_\eta}\left((\psi_x\psi_y)_x+\frac{1}{2}\left(\psi_y^2-\psi_x^2\right)_y+(\Gamma(-\psi))_y\right)\wt\phi\, d(x,y)=0,
\end{align*}
for all $\phi\in C^1_0(\0)$ and with $\wt\phi :=\phi\circ\Phi\in C^{1+\alpha}_0(\ov\0_\eta)$.
Hence, $h\in C^{1+\alpha}(\ov\0)$ is a weak solution of the height function formulation, cf. Definition \ref{D:1}.
We are now in the position to use the regularity result Theorem 5.1 in \cite{MM13x} which states  
that the distributional derivatives $\p_q^mh$ also belong to $C^{1+\alpha}(\ov\0)$ for all $m\geq1.$
Particularly, setting $m=1$, we find that $h_p$ is differentiable with respect to $q$ and $\p_q(h_p)=\p_p(h_q)\in C^\alpha(\ov\0).$ 
Exploiting   the fact that $h$ is a weak solution,  we see that the distributional derivatives 
\[
\p_q\left(\G+\frac{1+h_q^2}{2h_p^2}\right),\quad  \p_p\left(\G+\frac{1+h_q^2}{2h_p^2}\right)=\p_q\left( \frac{h_q}{h_p}\right)
\]
 belong both to $C^\alpha(\ov\0)\subset L_r(\0).$
Additionally,   $1+h_q^2\in C^{1+\alpha}(\ov\0)$ and regarding $\G$ as an element of $ W^1_r(\0)$, we obtain
 \[
 \frac{1 }{h_p^2}\in W^1_r(\0).
 \]
Because $h$ satisfies  \eqref{PBC} and recalling that $h_p$ is a bounded function,  \cite[Theorem 7.8]{GT01} implies that  $h_p\in W^1_r(\0).$
Hence,    $h\in C^{1+\alpha}(\ov\0)\cap W^2_r(\0)$ and  it is not difficult to see that $h$
satisfies the first equation of  \eqref{PB} in $L_r(\0)$, cf. \eqref{PF}.
This completes our arguments.
\end{proof}

We now state  our  main existence result.

\begin{thm}[Existence result]\label{T:MT} We fix $r\in(1,\infty)$ , $p_0\in(-\infty,0),$  
and define the H\"older exponent $\alpha:=(r-1)/r\in(0,1).$
 We also assume that the vorticity function $\gamma$ belongs to $L_r((p_0,0)).$ 
 
 Then,  there exists a {positive integer $N$} such that for each  {integer $n\geq N$}
 there exists a local real-analytic  curve $ {\cC_{n}}\subset C^{1+\alpha}(\ov\0)$   consisting only of strong  solutions of the problem
 \eqref{PB}-\eqref{PBC}. 
 Each solution $h\in {\cC_{n}}$, {$n\geq N,$}  satisfies additionally
 \begin{itemize}
 \item[$(i)$] $h\in W^2_r(\0)$,
 \item[$(ii)$] $h(\cdot,p)$ is a real-analytic map for all $p\in[p_0,0].$ 
 \end{itemize}
 Moreover, each curve ${\cC_{n}}$  contains a laminar flow solution
 and all the other points on the curve  describe waves  that have minimal period $2\pi/n $, only one crest and trough per period, and are symmetric with respect to the crest line.  
\end{thm}

\begin{rem}\label{R:0} 
While proving Theorem \ref{T:MT} we make no restriction on the constant $g$, meaning that the result  is true for capillary-gravity waves but also in the context of capillary waves (when we set $g=0$).

Sufficient conditions which allow  us to choose {$N=1$} in Theorem \ref{T:MT} can be found in Lemma \ref{L:9}.

 Also, if $\gamma\in C((p_0,0)),$ the solutions found in Theorem \ref{T:MT} are classical as one can easily show that, additionally to the regularity properties stated in Theorem \ref{T:EQ}, 
 we also  have $h\in C^2(\0),$ $\psi\in C^2(\0_\eta)$, and  $(u,v,P)\in C^1(\0_\eta)$.
\end{rem}

\section{Weak solutions for the height function formulation}\label{Sec:3}
This last section is dedicated to proving Theorem \ref{T:MT}.
Therefore, we pick $r\in(1,\infty)$ and let $\alpha=(r-1)/r\in(0,1)$ be fixed in the remainder of this paper.
The   formulation \eqref{PB} is very useful when trying to determine classical solution of the water wave problem \cite{W06b, W06a}.
However, when the vorticity function belongs to $L_r((p_0,0)),$  $r\in(1,\infty),$ the curvature term and the lack of regularity of the vorticity function gives rise to several difficulties when trying to consider
the equations \eqref{PB} in a suitable (Sobolev) analytic setting.
For example, the trivial solutions of \eqref{PB}, see Lemma \ref{L:LFS} below, belong merely to $W^2_r(\0)\cap C^{1+\alpha}(\ov\0).$
When trying to prove the Fredholm property of the linear operator associated to the linearization of the problem around these trivial solutions, 
one has to deal   with an elliptic equation in divergence form and having coefficients merely in $W^1_r(\0)\cap C^{\alpha}(\ov\0),$ cf. \eqref{L1}.
The solvability of elliptic boundary value problems in $W^2_r(\0)$ requires in general though more regularity from  the coefficients.  
Also, the trace $\tr_0 h_{qq}$ which appears in the second equation of \eqref{PB} is meaningless for functions in  $W^2_r(\0).$

Nevertheless, using the fact that the operator $(1-\p_q^2):H^2(\s)\to L_2(\s)$ is an isomorphism and   the divergence structure of the first equation of \eqref{PB}, that is
\[\left(\frac{h_q}{h_p}\right)_q-\left(\Gamma+\frac{1+h_q^2}{2h_p^2}\right)_p=0\qquad\text{in $\0$,}\]
with $\G$ being defined by the relation \eqref{E:G}, one can introduce the following definition of a weak solution of \eqref{PB}. 
\begin{defn}\label{D:1} A function $h\in C^{1}(\ov\0)$ which satisfies \eqref{PBC} is called a {\em weak solution} of \eqref{PB} if we have
\begin{subequations}\label{WF}
\begin{align}
h+(1-\p_q^2)^{-1}\tr_0\left( \frac{\left(1+h_q^2+(2gh-Q)h_p^2\right)(1+h_q^2)^{3/2}}{2\sigma h_p^2}-h\right)=&0\qquad\text{on $p=0$;}\label{PB0}\\[1ex]
h=&0\qquad\text{on $p=p_0$;}\label{PB1}
\end{align}
and if $h$ satisfies the following integral equation
\begin{equation}\label{PB2}
 \int_\0\frac{h_q}{h_p}\phi_q-\left(\Gamma+\frac{1+h_q^2}{2h_p^2}\right)\phi_p\,  d(q,p)=0\qquad\text{for all $\phi\in C^1_0(\0)$.}
\end{equation}
\end{subequations}
\end{defn}

Clearly, any strong solution $h\in C^{1+\alpha}(\ov\0)\cap W^2_r(\0)$ with $\tr_0 h \in W^2_r(\s)$ is a weak solution of \eqref{PB}.
Furthermore,  because of \eqref{PB0}, any weak solution of \eqref{PB} has additional regularity on the boundary component $p=0,$ that is $\tr_0 h\in C^2(\s).$
The arguments used in the last part of the proof of Theorem  \ref{T:EQ} show in fact that any weak solution $h$ which belongs to $C^{1+\alpha}(\ov\0)$ is a strong solution of \eqref{PB} (as stated in Theorem \ref{T:EQ} $(i)$).

The formulation \eqref{WF} has the advantage that in can be recast as an operator equation in a functional setting that enables us to use bifurcation results to  prove existence of weak solutions. 
To present this setting, we introduce the following Banach spaces:
\begin{align*}
 X&\!:=\!\left\{\wt h\in C^{1+\alpha}_{2\pi/n}(\ov\0)\,:\, \text{$\wt h$ is even in $q$ and $\wt h\big|_{p=p_0}=0$}\right\},\\
 Y_1&\!:=\!\{f\in\mathcal{D}'(\0)\,:\, \text{$f=\p_q\phi_1+\p_p\phi_2$ for  $\phi_1,\phi_2\in C^\alpha_{2\pi/n}(\ov\0)$ with $\phi_1$ odd and $\phi_2$ even in $q$}\},\\
 Y_2&\!:=\!\{\varphi\in C^{1+\alpha}_{2\pi/n}(\s)\,:\, \text{$\varphi$ is even}\},
\end{align*}
the positive integer  $n\in\N$  being fixed  later on.
The subscript $2\pi/n$ is used to express $2\pi/n-$periodic in $q$.
We recall that  $Y_1$ is a Banach space with the norm
\[
\|f\|_{Y_1}:=\inf\{\|\phi_1\|_\alpha+\|\phi_2\|_\alpha\,:\, f=\p_q\phi_1+\p_p\phi_2\}.
\]
In the following lemma we determine all laminar flow solutions of \eqref{WF}.
They correspond to waves with a flat surface $\eta=0$ and having parallel streamlines.

\begin{lemma}[Laminar flow solutions]\label{L:LFS} Let $\Gamma_{M}:=\max_{[p_0,0]}\Gamma$.
For every $\lambda\in(2\G_M,\infty)$, the function $H(\cdot;\lambda)\in W^2_r((p_0,0))$
with
\[
H(p;\lambda):=\int_{p_0}^p \frac{1}{\sqrt{\lambda-2\Gamma(s)}}\, ds\qquad\text{for $p\in[p_0,0]$}
\]
is a weak solution of \eqref{WF} provided that
\[
Q=Q(\lambda):=\lambda+2g\int_{p_0}^0\frac{1}{\sqrt{\lambda-2\Gamma(p)}}\, dp.
\]
There are no other weak solutions of \eqref{WF} that are independent of $q$.
\end{lemma}
\begin{proof}
 It readily follows from  \eqref{PB2} that if $H$ is a weak solution of \eqref{WF} that is independent of the variable  $q$, then $2\Gamma+1/H_p^2=0$ in $\mathcal{D}'((p_0,0)).$  
 The expression for $H$ is obtained now by using the relation \eqref{PB1}.
 When verifying the  boundary condition \eqref{PB0},  the relation $(1-\p_q^2)^{-1}\xi=\xi$ for all $\xi\in\R$ yields that $Q$ has to be equal with $Q(\lambda).$ 
\end{proof}

Because  $H(\cdot;\lambda)\in W^2_r((p_0,0))$, we can interpret by means of Sobolev's embedding $H(\cdot;\lambda)$ as being an element of  $X.$
We now are in the position of reformulating  the problem \eqref{WF} as an abstract operator  equation. 
Therefore, we introduce the nonlinear and nonlocal  operator $\cF:=(\cF_1,\cF_2):(2\G_M,\infty)\times X\to Y:=Y_1\times Y_2$ by the relations
\begin{align*}
 \cF_1(\lambda,{\wt h})\!:=\!&\left(\frac{{\wt h}_q}{H_p+{\wt h}_p}\right)_q-\left(\Gamma+\frac{1+{\wt h}_q^2}{2(H_p+{\wt h}_p)^2}\right)_p,\\
 \cF_2(\lambda,{\wt h})\!:=\!&\tr_0{\wt h}+(1-\p_q^2)^{-1}\tr_0\!\left(\!\frac{\left(1+{\wt h}_q^2+(2g(H+{\wt h})-Q)
 (H_p+{\wt h}_p)^2\right)(1+{\wt h}_q^2)^{3/2}}{2\sigma(H_p+{\wt h}_p)^2}-{\wt h}\!\right)
\end{align*}
for $(\lambda,{\wt h})\in (2 \G_M,\infty)\times X,$
whereby $H=H(\cdot;\lambda)$ and $Q=Q(\lambda)$ are defined in Lemma \ref{L:LFS}.
The operator $\cF$ is well-defined and   it depends real-analytically on its arguments, that is 
\begin{align}\label{BP0}
 \cF\in C^\omega((2\Gamma_M,\infty)\times X, Y).
\end{align}
With this notation, determining the weak solutions of the problem \eqref{PB} reduces to determining the zeros $(\lambda,{\wt h})$ of the   equation
\begin{align}\label{BP}
 \cF(\lambda,{\wt h})=0\qquad\text{in $Y$ }
\end{align}
for which ${\wt h}+H(\cdot;\lambda)$ satisfies \eqref{PBC}.
From the definition of $\cF$ we know that the laminar flow solutions of \eqref{PB} correspond  to the trivial solutions of $\cF$
\begin{align}\label{BP1}
 \cF(\lambda,0)=0\qquad\text{for all $\lambda\in(2 \G_M,\infty).$}
\end{align}
Actually, if   $(\lambda,{\wt h})$ is a solution of \eqref{BP},  the function $h:={\wt h}+H(\cdot;\lambda)\in X$ 
is a weak solution of \eqref{PB} when $Q=Q(\lambda)$, provided that ${\wt h}$ is sufficiently small in  $C^1(\ov\0).$  
In order to use the theorem on bifurcation from simple eigenvalues due
to Crandall and Rabinowitz \cite{CR71} in the setting of \eqref{BP}, we need to determine special values of $\lambda$ for which the Fr\'echet derivative  
$\p_{\wt h}\cF(\lambda,0)\in\kL(X,Y)$,
defined by
\[
\p_{\wt h}\cF(\lambda,0)[w]:=\lim_{\e\to0}\frac{\cF(\lambda,\e w)-\cF(\lambda,0)}{\e}\qquad\text{for $w\in X$,}
\]
is a Fredholm operator of index zero with a one-dimensional kernel.
To this end, we compute that
$\p_{\wt h} \cF (\lambda,0)=:(L,T)$ with $L\in\kL(X,Y_1) $ and $T\in\kL(X,Y_2)$  being given by
\begin{equation}\label{L1}
\begin{aligned}
 Lw:=& \left(\frac{w_q}{H_p}\right)_q+\left(\frac{w_p}{H_p^3}\right)_p,\\
 Tw:=&\tr_0 w+(1-\p_q^2)^{-1} \tr_0 \left(\frac{gw-\lambda^{3/2}w_p}{\sigma}-w\right)
\end{aligned}\qquad\quad\text{for $w\in X,$}
\end{equation}
and with $H=H(\cdot;\lambda)$ as in Lemma \ref{L:LFS}.

 We now study the properties of the linear operator $\p_{\wt h} \cF (\lambda,0)$,  $\lambda>2\G_M.$ 
Recalling that $H\in C^{1+\alpha}([p_0,0]),$ we obtain together with \cite[Theorem 8.34]{GT01}  the following result.
\begin{lemma}\label{L:2}
   The Fr\' echet  derivative  $\p_{\wt h} \cF(\lambda,0)\in\kL(X,Y)$ is a Fredholm operator of index zero for each $\lambda\in(2\G_M,\infty).$ 
\end{lemma}
\begin{proof}
See the proof of Lemma 4.1 in \cite{MM13x}.
\end{proof}

In order to apply the previously mentioned bifurcation result,   we need   to determine special values for $\lambda$ such that the kernel of  $\p_{\wt h} \cF(\lambda,0)$ is a  subspace of $X$ of dimension one.
To this end, we observe that if $0\neq w\in X$ belongs to the kernel of $\p_{\wt h} \cF(\lambda,0)$, the relation $Lw=0$ in $Y_1$ implies that, for each $k\in\N,$ the Fourier coefficient
\[
w_k(p):=\la w(\cdot, p)|\cos(kn\cdot)\ra_{L_2}:=\int_0^{2\pi} w(q,p)\cos(knq)\, dq\qquad\text{for $p\in[p_0,0]$}
\]
belongs to $C^{1+\alpha}([p_0,0])$ and solves the equation
\begin{align}\label{EQ:M}
 \left(\frac{w_k'}{H_p^3}\right)'-\frac{(kn)^2w_k}{H_p}=0\qquad\text{in $\mathcal{D}'((p_0,0)).$}
\end{align}
Additionally,    multiplying the relation $Tw=0$ by $\cos(knq)$ we determine, in virtue  of the symmetry of the operator $(1-\p_q^2)^{-1},$
that is
\begin{align*}
\la f|(1-\p_q^2)^{-1}g\ra_{L_2}=\la  (1-\p_q^2)^{-1}f|g\ra_{L_2}\qquad\text{for all $f,g\in L_2(\s)$},
\end{align*}
 a further relation 
\[(g+\sigma (kn)^2)w_k(0)=\lambda^{3/2}w_k'(0).\]
Finally, because of $w\in X$, we get $w_k(p_0)=0$.
Since $W^1_r((p_0,0))$ is an algebra for any $r\in(1,\infty),$ cf. \cite{A75}, it is easy to see that   $w_k$ belongs to $ W^2_r((p_0,0))$ and that it solves the system
\begin{equation}\label{E:m}
\left\{
\begin{array}{rlll}
  (a^3(\lambda) w')'-\mu a(\lambda)w&=&0 &\text{in $L_r((p_0,0))$,}\\
  (g+\sigma\mu)w(0)&=&\lambda^{3/2}w'(0),\\
  w(p_0)&=&0,
  \end{array}\right.
\end{equation}
when $\mu=(kn)^2.$
For simplicity, we set  $a(\lambda):=a(\lambda;\cdot):=\sqrt{\lambda-2\Gamma}\in W^1_r((p_0,0)).$

Our task is to determine special values for $\lambda$ with the property that the system \eqref{E:m} has nontrivial solutions, 
which form a one-dimensional subspace of $W^2_r((p_0,0))$,    { only for $\mu=n^2.$}
Therefore, given $(\lambda,\mu)\in(2 \G_M,\infty)\times[0,\infty),$  we introduce  the Sturm-Liouville type operator 
$R_{\lambda,\mu}:W^2_{r,0} \to L_r((p_0,0))\times  \R$   by 
\begin{equation*}
 R_{\lambda,\mu}w:=
 \begin{pmatrix}
  (a^3(\lambda) w')'-\mu a(\lambda)w\\
  (g+\sigma\mu)w(0)-\lambda^{3/2}w'(0)
 \end{pmatrix}\qquad\text{for $w\in W^2_{r,0},$}
\end{equation*}
whereby  $W^2_{r,0}:=\{w\in W^2_r((p_0,0))\,:\, w(p_0)=0\}.$
Additionally, for   $(\lambda,\mu)$ as above, we let $v_i\in W^2_r((p_0,0))$, with
 $v_i:=v_i(\cdot;\lambda,\mu)$, denote   the unique solutions of  the initial value problems 
\begin{equation}\label{ERU}
\left\{\begin{array}{lll}
  (a^3(\lambda) v_1')'-\mu a(\lambda)v_1=0\qquad \text{in $L_r((p_0,0))$},\\[1ex]
  v_1(p_0)=0,\quad v_1'(p_0)=1,
 \end{array} \right. 
 \end{equation}
and
\begin{equation}\label{ERUa}
 \left\{\begin{array}{lll}
  (a^3(\lambda)v_2')'-\mu a(\lambda)v_2=0\qquad \text{in $L_r((p_0,0))$},\\[1ex]
  v_2(0)=\lambda^{3/2},\quad v_2'(0)=g+\sigma\mu.
 \end{array}
 \right.
\end{equation}
Similarly as in the bounded vorticity case $\gamma\in L_\infty((p_0,0)) $ considered in \cite{MM13x}, we have the following property.
\begin{prop}\label{P:2} Given  $(\lambda,\mu)\in(2\G_M,\infty)\times[0,\infty),$ $R_{\lambda,\mu}$ is a Fredholm operator of index zero and its kernel is at most one-dimensional.
 Furthermore, the kernel of  $R_{\lambda,\mu}$ is one-dimensional exactly when  the functions  $v_i$, $i=1,2,$ given by \eqref{ERU} and \eqref{ERUa}, are linearly dependent.
 In the latter case we have $\ke R_{\lambda,\mu}=\spa\{v_1\}.$
\end{prop}
\begin{proof}
 First of all, $R_{\lambda,\mu}$ can be decomposed as the sum $R_{\lambda,\mu}=R_I+R_c$, whereby
  \[
 R_Iw:=
 \begin{pmatrix}
 (a^3(\lambda)w')'-\mu a(\lambda)w\\
  -\lambda^{3/2}w'(0)
 \end{pmatrix} \qquad \text{and}\qquad
R_cw:=
 \begin{pmatrix}
 0\\
  (g+\sigma\mu) w(0)
 \end{pmatrix} 
 \]
 for all $w\in W^2_{r,0}.$  
It is not difficult to see that $R_c$ is a compact operator.
Next, we show that   $R_I:W^2_{r,0} \to L_r((p_0,0))\times  \R$ is an isomorphism.
 Indeed, if $w\in W^2_{r,0}$ solves  the equation $R_Iw=(f,A), $ with $ (f,A)\in L_r((p_0,0))\times  \R$, then, since $W^2_r((p_0,0))\hookrightarrow C^{1+\alpha}([p_0,0]),$ we have
  \begin{equation}\label{VF}
  \int_{p_0}^0\left(a^3(\lambda)w'\varphi'+\mu a(\lambda)w\varphi\right)dp=-A\varphi(0)-\int_{p_0}^0 f\varphi\, dp
 \end{equation}
 for all $\varphi\in H_*:=\{w\in W^1_2((p_0,0))\,:\, w(p_0)=0\}$.
 The right-hand side of \eqref{VF} defines a linear functional in $\kL(H_*,\R)$ and that the left-hand side corresponds to a bounded bilinear and coercive functional in $H_*\times H_*.$
Therefore, the existence and uniqueness of a solution $w\in H_*$ of \eqref{VF} follows from the Lax-Milgram theorem, cf. \cite[Theorem 5.8]{GT01}.
In fact, one can easily see that  $w_*\in W^2_{r,0}$, so that  $R_I$ is  indeed an isomorphism. 

That the  kernel of  $R_{\lambda,\mu}$ is at most one-dimensional can be seen from the observation that if  $w_1,w_2\in W^2_r((p_0,0))$ are  solutions of  $(a^3(\lambda) w')'-\mu a(\lambda)w=0$, then
\begin{equation}\label{BV}a^3(\lambda)(w_1w_2'-w_2w_1')=const. \qquad\text{in $[p_0,0]$}.\end{equation}
Particularly, if  $w_1, w_2\in W^2_{r,0},$  we obtain, in view of $a(\lambda)>0 $ in $[p_0,0],$ that $w_1$ and $w_2$ are linearly dependent.
To finish the proof, we notice that if the functions $v_1$ and $v_2$, given by \eqref{ERU} and \eqref{ERUa}, are linearly dependent, then they both belong to $\ke R_{\lambda,\mu}.$
Moreover, if  $0\neq v\in \ke R_{\lambda,\mu},$ the relation \eqref{BV} yields that $v$ is collinear with both $v_1 $ and $v_2$, argument which completes our proof.
 \end{proof}

In view of the Proposition \ref{P:2}, we are left  to determine   $(\lambda,\mu)\in(2\G_M,\infty)\times[0,\infty)$ for which the Wronskian 
\[
W(p;\lambda,\mu):=\left|
\begin{array}{lll}
 v_1&v_2\\
 v_1'&v_2'
\end{array}
\right|
\]
vanishes on the entire interval $[p_0,0].$ 
Recalling \eqref{BV}, we arrive at the problem of determining the zeros of the real-analytic (\eqref{ERU} and \eqref{ERUa} can be seen as initial value problems for first order  ordinary differential equations)   function
$W(0;\cdot,\cdot):(2\G_M,\infty)\times  [0,\infty)\to\R$ defined by 
\begin{equation}\label{DEFG}
W(0;\lambda,\mu):=\lambda^{3/2}v_1'(0;\lambda,\mu)-(g+\sigma\mu)v_1(0;\lambda,\mu).
\end{equation}
We emphasize that the methods used in \cite{CM13xx, MM13x, W06b, W06a} in order to study the solutions 
of $W(0;\cdot,\cdot)=0$ cannot be used for general $L_r-$integrable vorticity functions.
Indeed, the  approach {chosen in the context of classical $C^{2+\alpha}-$solutions} in \cite{W06b, W06a}
is based on regarding the Sturm-Liouville problem \eqref{E:m} as a non standard   eigenvalue problem (the boundary condition depends on the eigenvalue $\mu$).
For this, the  author of  \cite{W06b, W06a} introduces a Pontryagin space with a indefinite inner product and uses abstract results pertaining to this setting.
In our context such considerations are possible only when restricting  $r\geq 2.$
On the other hand, the methods used  in  \cite{CM13xx, MM13x} are based on direct estimates for the solution of \eqref{ERU}, but these estimates 
rely to a large extent on the boundedness of $\gamma.$      
Therefore,  we need to find a new approach when allowing for general  $L_r-$integrable vorticity functions.

{Our strategy is as follows: in a first step we find a constant $\lambda_0\geq 2\G_M $ such 
that the function $ W(p;\lambda,\cdot)$ changes sign on $(0,\infty)$ for all $\lambda>\lambda_0,$ 
cf. Lemmas \ref{L:1} and \ref{L:4}.  
For this, the estimates established in Lemma \ref{L:3} within the setting of ordinary differential equations are crucial.
In a second step, cf. Lemmas \ref{L:5} and \ref{L:6}, we prove that $ W(p;\lambda,\cdot)$ changes  sign exactly once  on $(0,\infty)$, the particular value where $ W(p;\lambda,\cdot)$
vanishes being called  $\mu(\lambda).$
The properties of the mapping $\lambda\mapsto \mu(\lambda)$  derived in Lemma \ref{L:6}   
are the core of the analysis of the kernel of $\p_{\wt h}\mathcal{F}(\lambda,0).$
}

As  a first result, we state the following lemma.
\begin{lemma}\label{L:1}
 There exists a unique minimal $\lambda_0\geq 2\G_M$ such that $W(0;\lambda,0)>0$ for all $\lambda>\lambda_0.$
\end{lemma}
\begin{proof}
First, we note that given  $(\lambda,\mu)\in(2\G_M,\infty)\times[0,\infty)$, the function $v_1$  satisfies  the following integral relation
\begin{equation}\label{v1}
 v_1(p)=\int_{p_0}^p\frac{a^3(\lambda;p_0)}{a^3(\lambda;{s})}\, d{s}+\mu\int_{p_0}^p\frac{1}{a^3(\lambda;s)}\int_{p_0}^sa(\lambda;r)v_1(r)\, dr\, ds\qquad\text{for $p\in[p_0,0].$}
\end{equation}
Particularly, $v_1$ is a strictly increasing function on $[p_0,0]$. 
Furthermore, since $a(\lambda;0)=\lambda^{1/2},$ we get
\begin{align*}
 W(0;\lambda,0)=&a^3(\lambda;p_0)-g\int_{p_0}^0\frac{a^3(\lambda;p_0)}{a^3(\lambda;p)}\, dp=a^3(\lambda;p_0)\left(1-g\int_{p_0}^0\frac{1}{a^3(\lambda;p)}\, dp\right)\to_{\lambda\to\infty}\infty.
\end{align*}
This proves the claim.
\end{proof}

We note that if $g=0,$ then $\lambda_0=2\G_M.$
In the context of capillary-gravity water waves it is possible  to choose,  in the case of a  bounded vorticity function,  $\lambda_0>2\G_M$ as being the unique solution of the equation  $W(0;\lambda_0,0)=0$.
In contrast, for certain unbounded vorticity functions $\gamma\in L_r((p_0,0)),$ with $r\in(1,\infty),$ the latter equation has no zeros in $(2\G_M,\infty).$
Indeed, if we set $\gamma(p):=\delta(-p)^{-1/(kr)}$ for $p\in(p_0,0),$ where $\delta>0$ and $k,r\in(1,3) $ satisfy $kr<3,$ then $\gamma\in L_r((p_0,0))$ and, for sufficiently large $\delta $ (or small $p_0$), we have  
 \begin{align*}
  \inf_{\lambda>2\G_M} W(0;\lambda,0)>0.
 \end{align*}
This property leads to restrictions  on the wavelength of the water waves bifurcating from the laminar flow solutions found in Lemma \ref{L:LFS}, cf. Proposition \ref{P:3}.

The  estimates below will be used in Lemma \ref{L:4} to bound  the integral mean  and the first order moment  of the    solution $v_1$ of \eqref{ERU} on intervals  $[p_1(\mu),0]$ with $p_1(\mu)\nearrow0$ as $\mu\to\infty.$
\begin{lemma}\label{L:3}
 Let $p_1\in(p_0,0)$, $A, B\in(0,\infty),$ and $(\lambda,\mu)\in(2\G_M,\infty)\times[0,\infty)$ be fixed  and define the positive constants
 \begin{equation}\label{constants}
 \begin{aligned}
  &\un C:=\min_{p\in[p_1,0]}\frac{a^3(\lambda;p_1)}{a^3(\lambda;p)},\quad \ov C:=\max_{p\in[p_1,0]}\frac{a^3(\lambda;p_1)}{a^3(\lambda;p)},\\
  &\un D:=\min_{s,p\in[p_1,0]}\frac{a(\lambda;s)}{a^3(\lambda;p)},
  \quad \ov D:=\max_{s,p\in[p_1,0]}\frac{a(\lambda;s)}{a^3(\lambda;p)}.
 \end{aligned}
 \end{equation}
Then, if  $v\in W^2_r((p_1,0))$  is the solution of 
\begin{equation}\label{EEE}
\left\{\begin{array}{lll}
  (a^3(\lambda) v')'-\mu a(\lambda)v=0\qquad \text{in $L_r((p_1,0))$},\\[1ex]
  v(p_1)=A,\quad v'(p_1)=B,
 \end{array}
 \right.
\end{equation}
we have the following estimates
\begin{align}
 \int_{p_1}^0v(p)\, dp\!\leq& -\frac{A\mu^{-1/2}\sinh(p_1\sqrt{\ov D}\mu^{1/2})}{\sqrt{\ov D}}+\frac{B\ov C\mu^{-1}\left(\cosh(p_1\sqrt{\ov D}\mu^{1/2})-1\right)}{\ov D},\label{FE1}\\[1ex]
 \int_{p_1}^0(-p)v(p)\, dp\!\geq& \frac{A\mu^{-1 }\left(\cosh(p_1\sqrt{\un D}\mu^{1/2})-1\right)}{\un D}+\frac{B\un C\mu^{-1}\!\left(\sqrt{\un D}p_1 - \sinh(p_1\sqrt{\un D}\mu^{1/2}) \mu^{-1/2}\right)}{\un D^{3/2}}\label{FE2}.
\end{align}
\end{lemma}
\begin{proof}
 It directly follows  from \eqref{EEE} that
 \begin{align}\label{Der}
  v'(p)=\frac{a^3(\lambda;p_1)}{a^3(\lambda;p)}B+\mu\int_{p_1}^p\frac{a (\lambda;s)}{a^3(\lambda;p)}v(s)\, ds\qquad\text{for all $p\in[p_1,0],$}
 \end{align}
and therefore 
\[
v'(p)\leq B\ov C+\mu \ov D\int_{p_1}^pv(s)\, ds \qquad\text{in $p\in[p_1,0],$}
\]
cf. \eqref{constants}.
Letting now $\ov u:[p_0,0]\to\R$ be the function defined by
\[
\ov u(p):=\int_{p_1}^pv(s)\, ds\qquad\text{for $p\in[p_1,0],$}
\]
we find that $\ov u\in W^3_r((p_0,0))$ solves the following problem
\begin{equation*}
 \ov  u''-\mu \ov D\ov  u\leq B\ov C\quad \text{in $(p_1,0)$},\qquad \ov u(p_1)=0,\, \, \ov u'(p_1)=A.
\end{equation*}
It is not difficult to see that $\ov  u\leq \ov z$ on $[p_1,0],$ where $\ov z$ denotes the  solution of the initial value problem 
\begin{equation*}
  \ov  z''-\mu \ov D \ov z= B\ov C\quad \text{in $(p_1,0)$},\qquad \ov  z(p_1)=0,\, \, \ov z'(p_1)=A.
\end{equation*}
The solution $\ov z$ of this  problem can be determined explicitly 
\begin{align*}
 \ov z(p)=\frac{A\sinh(\sqrt{\ov D}\mu^{1/2}(p-p_1))}{\sqrt{\ov D\mu}}+\frac{B\ov C\left(\cosh(\sqrt{\ov D}\mu^{1/2}(p-p_1))-1\right)}{\ov D\mu}, \qquad p\in[p_1,0], 
\end{align*}
which gives, in virtue of $\ov u(0)\leq \ov z(0),$ the first estimate \eqref{FE1}.

In order to prove the second estimate \eqref{FE2},  we first note that integration by parts leads us to
\begin{align*}
\int_{p_1}^0(-p)v(p)\,dp =\int_{p_1}^0\int_{p_1}^p v(s)\, ds\,dp \qquad\text{in $[p_1,0],$}
\end{align*}
so that it is natural to define the function $\un u:[p_0,0]\to\R$  by the relation
\[
\un u(p):=\int_{p_1}^p\int_{p_1}^rv(s)\, ds\, dr\qquad\text{for $p\in[p_1,0].$}
\]
Recalling \eqref{Der}, we find similarly as before that
\[
v'(p)\geq B\un C+\mu \un D\int_{p_1}^pv(s)\, ds \qquad\text{in $p\in[p_1,0],$}
\]
and integrating this inequality over $(p_1,p)$, with $p\in(p_1,0)$, we get
\begin{align*}
 v(p)\geq A+B\un C(p-p_1)+\mu \un D\int_{p_1}^p\int_{p_1}^rv(s)\, ds\, dr \qquad\text{in $p\in[p_1,0].$}
\end{align*}
Whence,  $\un u\in W^4_r((p_0,0))$ solves the  problem
\begin{equation*}
 \un  u''-\mu \un D\,\un  u\geq A+B\un C(p-p_1)\quad \text{in $(p_1,0)$},\qquad \un u(p_1)=0,\, \, \un u'(p_1)=0.
\end{equation*}
As the right-hand side of the  above inequality is positive, we find   that $\un  u\geq \un z$ on $[p_1,0],$ where $\un z$ stands now for the  solution of the   problem 
\begin{equation*}
  \un  z''-\mu \un D \, \un z=  A+B\un C(p-p_1)\quad \text{in $(p_1,0)$},\qquad \un  z(p_1)=0,\, \, \un z'(p_1)=0.
\end{equation*}
One can easily verify that $\un z$   has the following expression 
\begin{align*}
 \un z(p)=&\frac{A\left(\cosh(\sqrt{\un D}\mu^{1/2}(p-p_1))-1\right)}{ \un D\mu }\\
 &+\frac{B\un C\left( \un D^{-1/2}\sinh(\sqrt{\un D}\mu^{1/2}(p-p_1))\mu^{-1/2}-(p-p_1)\right)}{\un D\mu}
\end{align*}
for $ p\in[p_1,0] $, and, since  $\un u(0)\geq \un z(0),$ we obtain the desired estimate \eqref{FE2}.
\end{proof}

The estimates \eqref{FE1} and \eqref{FE2} are  the main tools when proving the following result.
\begin{lemma}\label{L:4}
 Given $\lambda> 2\G_M,$ we have that
 \begin{align}\label{ES}
  \lim_{\mu\to\infty} W(0;\lambda,\mu)=-\infty.
 \end{align}
\end{lemma} 
\begin{proof}
 Recalling the relations \eqref{DEFG} and \eqref{v1}, we write  $W(0;\lambda,\mu)=T_1+\mu T_2,$ 
 whereby we defined
 \begin{align*}
 T_1&:=a^3(\lambda;p_0)\left(1-(g+\sigma\mu)\int_{p_0}^0\frac{1}{a^3(\lambda;p)}\, dp\right),\\
 T_2&:=\int_{p_0}^0a(\lambda;p)v_1(p)\, dp-(g+\sigma\mu)\int_{p_0}^0\frac{1}{a^3(\lambda;s)}\int_{p_0}^sa(\lambda;r)v_1(r)\, dr\, ds.
\end{align*}
Because $a(\lambda)$ is a continuous and positive function that does on depend on $\mu$, it is easy to see that $T_1\to-\infty$ as $\mu\to\infty.$
In the remainder of this proof we show that 
\begin{equation}\label{QE1}
 \lim_{\mu\to\infty} T_2=-\infty.
\end{equation}
In fact,  since  $a(\lambda)$ is bounded from below and from above in $(0, \infty)$,  we see, by using  integration by parts, that \eqref{QE1} holds provided that  there exists a constant $\beta\in(0,1)$ such that
\begin{equation}\label{QE2}
\lim_{\mu\to\infty} \left( \int_{p_0}^0v_1(p)\, dp-\mu^{\beta}\int_{p_0}^0(-p)v_1(p)\, dp\right)=-\infty.
\end{equation}
We now fix  $\beta\in(1/2,1)$ and prove that \eqref{QE2} is satisfied if we make this choice for $\beta.$ 
Therefore, we first choose $\gamma\in(1/2,\beta)$ with
\begin{align}\label{ch}
 \frac{2\beta-1}{2\gamma-1}=4.
\end{align}
Because for sufficiently large $\mu$ we have 
\begin{align*}
\int_{p_0}^{-\mu^{-\gamma}}v_1(p)\, dp-\mu^{\beta}\int_{p_0}^{-\mu^{-\gamma}}(-p)v_1(p)\, dp\leq&\int_{p_0}^{-\mu^{-\gamma}}v_1(p)\, dp-\mu^{\beta}\int_{p_0}^{-\mu^{-\gamma}}\mu^{-\gamma}v_1(p)\, dp\\[1ex]
=&(1-\mu^{\beta-\gamma})\int_{p_0}^{-\mu^{-\gamma}}v_1(p)\, dp\to_{\mu\to\infty}-\infty,
\end{align*}
we are left to show that
\begin{align}\label{QE3}
\limsup_{\mu\to\infty}\left(\int_ {-\mu^{-\gamma}}^0v_1(p)\, dp-\mu^{\beta}\int_{-\mu^{-\gamma}}^0(-p)v_1(p)\, dp\right)<\infty.
\end{align}
The difficulty of showing \eqref{QE2} is mainly caused by the fact that the function $v_1$ grows very fast with $\mu.$
However, because the volume of the interval of integration in \eqref{QE3} decreases also very fast when $\mu\to\infty$, the estimates derived in Lemma \ref{L:3} are accurate enough to establish \eqref{QE3}.
To be precise, for all $\mu>(-1/p_0)^{1/\gamma}$, we set $p_1:=-\mu^{-\gamma}$, $A:=v_1(p_1),$ $B:=v_1'(p_1)$, and obtain that the solution $v_1$ of \eqref{EEE} satisfies
\begin{align}\label{QE4}
\int_ {-\mu^{-\gamma}}^0v_1(p)\, dp-\mu^{\beta}\int_{-\mu^{-\gamma}}^0(-p)v_1(p)\, dp\leq \frac{A \sinh(\sqrt{\ov D}\mu^{1/2-\gamma})}{\un D\mu^{1/2}}E_1+\frac{B\un C}{\un D\mu}E_2,
\end{align}
whereby $A, B, \ov C,\un C,\ov D,\un D$ are functions of $\mu$ now, cf. \eqref{constants}, and 
\begin{align*}
 E_1&:=\frac{\un D}{\sqrt{\ov D}}- \mu^{\beta-1/2 }\frac{ \cosh( \sqrt{\un D}\mu^{1/2-\gamma})-1 }{ \sinh(\sqrt{\ov D}\mu^{1/2-\gamma})},\\[1ex]
 E_2&:= \frac{\ov C\un D}{\un C\ov D} \left(\cosh(\sqrt{\ov D}\mu^{1/2-\gamma})-1\right) -\mu^{\beta -\gamma}\left(\frac{\sinh(\sqrt{\un D}\mu^{1/2-\gamma})}{\sqrt{\un D}\mu^{1/2-\gamma}}-1\right).
\end{align*}
Recalling that  $\gamma>1/2$ and that   $A$,$B,$ $\ov C,\un C,\ov D,\un D$  are all   positive, it suffices to show that $E_1$ and $E_2$ are negative when $\mu$ is large.
In order to prove this property, we infer from \eqref{constants} that, as $\mu\to\infty,$ we have
\[
\ov D\to \lambda^{-1},\qquad \ov D\to \lambda^{-1},\qquad \ov C\to 1, \qquad \un C\to1.
\]
Moreover, using the substitution $t:=\sqrt{\un D}\mu^{1/2-\gamma} $ and l'Hospital's rule, we find 
\begin{align*}
 \lim_{\mu\to\infty}E_1=&1-\lim_{\mu\to\infty}\mu^{\beta-1/2 }\frac{ \cosh( \sqrt{\un D}\mu^{1/2-\gamma})-1 }{ \sinh(\sqrt{\un D}\mu^{1/2-\gamma})}\frac{\sinh(\sqrt{\un D}\mu^{1/2-\gamma})}{ \sinh(\sqrt{\ov D}\mu^{1/2-\gamma})}\\[1ex]
 =&1-\lim_{\mu\to\infty}\mu^{\beta-1/2 }\frac{ \cosh( \sqrt{\un D}\mu^{1/2-\gamma})-1 }{ \sinh(\sqrt{\un D}\mu^{1/2-\gamma})} =1-\frac{1}{\lambda^2} \lim_{t\searrow0}\frac{ \cosh( t)-1 }{ t^4\sinh(t)}=-\infty,
\end{align*}
cf. \eqref{ch}, and by similar  arguments
\begin{align*}
 \lim_{\mu\to\infty}E_2=&-\lim_{\mu\to\infty}\mu^{\beta -\gamma}\left(\frac{\sinh(\sqrt{\un D}\mu^{1/2-\gamma})}{\sqrt{\un D}\mu^{1/2-\gamma}}-1\right)=-\frac{1}{\lambda^{3/2}} \lim_{t\searrow0}\frac{ \sinh( t)-t }{ t^4}=-\infty.
\end{align*}
Hence, the right-hand side of \eqref{QE4} is negative when $\mu$ is sufficiently large, fact which proves the desired inequality  \eqref{QE3}.
\end{proof}

Combining the Lemmas \ref{L:1} and \ref{L:4}, we see that the equation $W(0;\cdot,\cdot)=0$ has at least a  solution for each $\lambda>\lambda_0.$
Concerning the sign of the first order derivatives  $W_\lambda(0;\cdot,\cdot)$ and $W_\mu(0;\cdot,\cdot)$
at the zeros    of $W(0;\cdot,\cdot)$, which will be used below to show that $W(0;\cdot,\cdot)$ has a unique zero  for each $\lambda>\lambda_0$, the
results established for a H\"older continuous   \cite{W06b, W06a} or for a bounded vorticity function \cite{CM13xx, MM13x} extend also to the case of a $L_r$-integrable vorticity function, without making any restriction on $r\in(1,\infty).$   

\begin{lemma}\label{L:5} Assume that $(\ov\lambda,\ov\mu)\in(\lambda_0,\infty)\times(0,\infty)$ satisfies $W(0; \ov\lambda,\ov\mu)=0.$
Then, we have
\begin{align}\label{slim}
 W_\lambda(0; \ov\lambda,\ov\mu)>0\qquad\text{and}\qquad W_\mu(0; \ov\lambda,\ov\mu)<0.
\end{align}
\end{lemma}
\begin{proof}
 The  Proposition \ref{P:2} and the discussion following it show that  
 $\ke R_{\ov\lambda,\ov\mu} =\spa\{v_1\}$,  whereby $v_1:=v_1(\cdot;\ov\lambda,\ov\mu)$. 
To prove the first claim, we note  that the algebra property of $W^1_r((p_0,0))$ yields that 
the partial derivative $v_{1,\lambda}:=\p_\lambda v_{1}(\cdot,\ov\lambda,\ov\mu)$ belongs to $W^2_r((p_0,0))$ and solves the problem
\begin{equation}\label{v1l}
 \left\{\begin{array}{lll}
 (a^{3}(\ov\lambda)v_{1,\lambda}')'-\ov\mu a(\ov\lambda) v_{1,\lambda}= 
 -(3a^2(\ov\lambda)a_{\lambda}(\ov\lambda)v_1')'+\ov\mu a_{\lambda}(\ov\lambda)v_1\qquad\text{in $ L_{r}((p_0,0))$,}\\[1ex]
  v_{1,\lambda}(p_0)= v_{1,\lambda}'(p_0)=0,
\end{array}\right.
\end{equation}
where $a_{\lambda}(\ov\lambda)=1/(2a(\ov\lambda))$.
Because  of the embedding $W^2_r((p_0,0))\hookrightarrow C^{1+\alpha}([p_0,0]),$ we find, by 
multiplying the differential equation satisfied by $v_1$, cf. \eqref{ERU}, with $v_{1,\lambda}$
and  the first equation of \eqref{v1l} with $v_1$, and after subtracting the resulting relations the first claim of \eqref{slim}
\begin{align*}
W_{\lambda}(0;\ov\lambda,\ov\mu)&=\ov\lambda^{3/2}v_{1,\lambda}'(0)+\frac{3}{2}\ov\lambda^{1/2}v_1'(0)-(g+\sigma\ov\mu)v_{1,\lambda}(0)\\
&=\frac{1}{v_1(0)}\left(\int_{p_0}^{0}\frac{3a(\ov\lambda)}{2} v_1'^{ 2}+\frac{\ov\mu}{2a(\ov\lambda)}v_1^2\, dp\right)>0.\end{align*}

For the second claim, we find as above that  $v_{1,\mu}:=\p_\mu v_{1}(\cdot,\ov\lambda,\ov\mu)\in W^2_r((p_0,0))$  
is the unique solution of the problem
\begin{equation}\label{v1mu}
 \left\{\begin{array}{lll}(a^{3}(\ov\lambda)v_{1,\mu}')'-\ov\mu a(\ov\lambda) v_{1,\mu}=a(\ov\lambda)v_1\qquad \text{in $ L_{r}((p_0,0))$},\\[1ex]
 v_{1,\mu}(p_0)=v_{1,\mu}'(p_0)=0.
 \end{array}\right.
\end{equation}
Also, if we   multiply the differential equation satisfied by $v_1$  with $v_{1,\mu}$ 
and  the first equation of \eqref{v1mu} with $v_1$, we get after building the difference of these relations 
\begin{align*}
 \int_{p_0}^0\!a(\ov\lambda)v_1^2\, dp=\ov\lambda^{3/2}\!v_{1,\mu}'(0)v_1(0)-\ov\lambda^{3/2}\!v_1'(0)v_{1,\mu}(0)=v_1(0)\left(\!\ov\lambda^{3/2}v_{1,\mu}'(0)-(g+\sigma\ov \mu)v_{1,\mu}(0)\!\right)\!,
\end{align*}
the last equality being a consequence of the fact that $v_1$ and  $v_2:=v_2(\cdot;\ov\lambda,\ov\mu) $ are collinear  for this choice of the parameters.
Therefore, we have
\begin{align}\label{qqq}
 W_\mu(0;\ov\lambda,\ov\mu)=\ov\lambda^{3/2}v_{1,\mu}'(0)-\sigma v_1(0)-(g+\sigma\ov\mu)v_{1,\mu}(0)=\frac{1}{v_1(0)}  \left(\int_{p_0}^0a(\ov\lambda)v_1^2\, dp-\sigma v_1^2(0)\right).
\end{align}
In order to determine the sign of the latter expression, we multiply the first equation of \eqref{ERU} by $v_1$
and get, by using once more the collinearity of $v_1$ and $v_2,$ that
\[
\int_{p_0}^0a(\ov\lambda)v_1^2\, dp-\sigma v_1^2(0)=\frac{1}{\ov\mu}\left(gv_1^2(0)-\int_{p_0}^0a^3(\ov\lambda)v_1'^2\, dp\right).
\]
If $g=0$, the latter expression is negative and we are done. 
On the other hand, if we consider gravity effects, because of $\ov\mu>0,$ it is easy to see that  $a^{3/2}(\ov\lambda)v_1'$ and $a^{-3/2}(\ov\lambda)$ 
are linearly independent functions, fact which ensures together with Lemma \ref{L:1} and with H\"older's inequality  that
\begin{align*}
gv_1^2(0)&=g\left(\int_{p_0}^0a^{3/2}(\ov\lambda)v_1'\frac{1}{a^{3/2}(\ov\lambda)}\, dp\right)^2\\
&<g\left(\int_{p_0}^0a^{3 }(\ov\lambda)v_1'^2 \, dp\right)\left(\int_{p_0}^0 \frac{1}{a^{3 }(\ov\lambda)}\, dp\right)\leq \int_{p_0}^0a^{3 }(\ov\lambda)v_1'^2 \, dp,
\end{align*}
and the desired claim follows from \eqref{qqq}.
\end{proof}

We conclude with the following result.
\begin{lemma}\label{L:6}
 Given $\lambda>\lambda_0,$ there exists a unique zero $\mu=\mu(\lambda)\in (0,\infty)$ of the equation  $W(0;\lambda,\mu(\lambda))=0.$
 The function 
 \[\mu:(\lambda_0,\infty)\to(\inf_{(\lambda_0,\infty)}\mu(\lambda),\infty),\qquad \lambda\mapsto\mu(\lambda)\] is strictly increasing, real-analytic, and bijective. 
\end{lemma}
\begin{proof}
Given $\lambda>\lambda_0,$ it follows from the Lemmas \ref{L:1} and \ref{L:4} that there exists a constant $\mu(\lambda)>0$ such that $W(0;\lambda,\mu(\lambda))=0.$
The uniqueness of this constant, and  the real-analyticity and the monotonicity of $\lambda\mapsto\mu(\lambda)$ follow readily from Lemma \ref{L:5} and the implicit function theorem. 
To complete the proof, let us assume that we found a sequence  $\lambda_n\to\infty$ such that $(\mu(\lambda_n))_n$ is bounded. 
Denoting by $v_{1n}$ the (strictly increasing) solution of \eqref{ERU} when $(\lambda,\mu)=(\lambda_n,\mu(\lambda_n)),$ we infer from \eqref{v1} that there exists a constant $C>0$ such that
\[
v_{1n}(p)\leq C\left(1+\int_{p_0}^pv_{1n}(s)\, ds\right)\qquad\text{for all $n\geq 1$ and $p\in[p_0,0].$}
\]
Gronwall's inequality yields  that the sequence $(v_{1n})_n$ is bounded in $C([p_0,0])$ and, together with \eqref{v1}, we find that
\[0=W(0;\lambda_n,\mu(\lambda_n))\geq a^3(\lambda_n;p_0)-(g+\sigma\mu(\lambda_n))v_{1n}(0)\underset{n\to\infty}\to\infty.\]
This is a contradiction, and the proof is complete.
\end{proof}

We choose now the integer $N$ from Theorem \ref{T:MT},   to be {the smallest positive integer which}  satisfies 
\begin{equation}\label{eq:rest}
{ N^2}>\inf_{(\lambda_0,\infty)}\mu(\lambda).
\end{equation}
Invoking Lemma \ref{L:6}, we  {find   a sequence $(\lambda_n)_{n\geq N}\subset (\lambda_0,\infty)$ having the properties that $\lambda_n\nearrow\infty$ and  
\begin{equation}\label{eq:sec}
 \text{$\mu(\lambda_n)=n^2$ \qquad for all $n\geq N.$}
\end{equation}}

We conclude the previous analysis with the following result.
\begin{prop}\label{P:3}
 Let {$N\in\N$ be  defined by \eqref{eq:rest}. 
 Then, for each $n\geq N $, the Fr\'echet derivative  
 $\p_{\wt h} \cF(\lambda_n,0)\in\kL(X,Y)$, with $\lambda_n$ defined by \eqref{eq:sec}, 
 is a Fredholm operator of index zero with a   one-dimensional kernel $\ke\p_{\wt h} \cF(\lambda_n,0)=\spa\{w_n\}$,
 whereby $w_n\in X$ is the function 
 $w_n(q,p):=v_1(p;\lambda_n,n^2)\cos(nq)$   for all $(q,p)\in\ov\0.$}
\end{prop}
\begin{proof}
 The result is a consequence of the  Lemmas \ref{L:2} and \ref{L:6}, and of  Proposition \ref{P:2}.
\end{proof}

In order to apply the theorem on bifurcations from simple eigenvalues to the   equation \eqref{BP}, we still have to verify the transversality condition
\begin{equation}\label{eq:TC}
 \p_{\lambda {\wt h}}\cF(\lambda_n,0)[w_n]\notin\im \p_{\wt h} \cF(\lambda_n,0)
\end{equation}
 for   $n\geq N.$ 
 \begin{lemma}\label{L:TC}
  The transversality condition  \eqref{eq:TC} is satisfied for all {$n\geq N$}.
 \end{lemma}
\begin{proof}
 The proof is similar to that of the Lemmas 4.4 and 4.5 in \cite{MM13x}, and therefore we omit it.
\end{proof}

We come to the proof of our main existence result.
\begin{proof}[Proof of Theorem \ref{T:MT}]
{Let $N$ be    defined by \eqref{eq:rest},       and let $(\lambda_n)_{n\geq N}\subset (\lambda_0,\infty)$
be the sequence defined by \eqref{eq:sec}.}
 Invoking the relations \eqref{BP0}, \eqref{BP1}, the Proposition \ref{P:3}, and the Lemma \ref{L:TC}, we see 
 that all the assumptions of the theorem on bifurcations from simple eigenvalues of Crandall and Rabinowitz \cite{CR71} 
 are satisfied for the equation \eqref{BP} at each of the points $\lambda=\lambda_n,$  $n\geq N.$
 Therefore,    for each  $n\geq N$, there exists   $\e_n>0$ and a real-analytic  curve 
\[\text{$(\wt \lambda_n,{\wt h}_n):(\lambda_n-\e_n,\lambda_n+\e_n)\to  (2\G_M,\infty)\times X,$  }\]
consisting only of solutions of the problem
 \eqref{BP}.
Moreover, as $s\to0$,  we have that
\begin{equation}\label{asex}
\wt\lambda_n(s)=\lambda_n+O(s)\quad \text{in $\R$},\qquad {\wt h}_n(s)=sw_n+O(s^2)\quad \text{in $X$},
\end{equation}
whereby $w_n\in X$ is the function defined in Proposition \ref{P:3}.
Furthermore, in a neighborhood of $(\lambda_n,0),$ the solutions of \eqref{BP} are either laminar or are located on the local curve $(\wt\lambda_n,{\wt h}_n)$.
The constants $\e_n$ are chosen sufficiently small to guarantee that $H(\cdot;\wt\lambda_n(s))+{\wt h}_n(s)$ satisfies \eqref{PBC} for all $|s|<\e_n$ and all $n\geq N.$
For each integer $n\geq N,$ the curve ${\cC_{n}}$ mentioned in Theorem \ref{T:MT} is
parametrized by $[s\mapsto H(\cdot;\wt\lambda_n(s))+{\wt h}_n(s)]\in C^\omega((-\e_n,\e_n),X).$

We pick now a function $h$ on one of the  local curves ${\cC_{n}}$.
In order to show that this  weak solution   of \eqref{WF} belongs to $ W^2_r(\0)$, we first infer from Theorem 5.1 in \cite{MM13x} 
that the distributional derivatives $\p_q^mh$ also belong to $C^{1+\alpha}(\ov\0)$ for all $m\geq1.$
Using the same arguments as in the last part of the proof of Theorem \ref{T:EQ}, we find that   $h\in C^{1+\alpha}(\ov\0)\cap W^2_r(\0)$ 
satisfies the first equation of  \eqref{PB} in $L_r(\0)$.
Because $(1-\p_q^2)^{-1}\in \kL(C^\alpha(\s), C^{2+\alpha}(\s))$, the equation \eqref{PB0} yields that $\tr_0 h\in C^{2+\alpha}(\s)$, and 
therefore $h$ is a strong solution of  \eqref{PB}.
Moreover,   by  \cite[Corollary 5.2]{MM13x}, result which shows that the regularity 
properties of the streamlines  of classical solutions  \cite{Hen10, DH12} persist even for weak solutions with merely integrable vorticity,   
$[q\mapsto h(q,p)]$ is a real-analytic map  for any $p\in[p_0,0]$.
Finally, because of \eqref{asex}, it is not difficult to see that any 
solution $h=H(\cdot;\wt \lambda_n(s))+{\wt h}_n(s)\in{\cC_{n}},$ with $s\neq0 $  sufficiently small, corresponds
to waves that possess a single crest per period  and which are 
symmetric with respect to the crest (and trough) line.
\end{proof}

As noted in the discussion following Lemma \ref{L:1}, when $r\in(1,3),$   
there are examples of vorticity functions $\gamma\in L_r((p_0,0))$ for which the mapping $\lambda\mapsto\mu(\lambda)$ defined in Lemma \ref{L:6} is
bounded away from zero on
$(\lambda_0,\infty)$.
This property imposes restrictions (through the { positive integer $N$}) on the wave length of the water waves solutions bifurcating from the laminar flows, 
cf. Theorem \ref{T:MT}.

The lemma below gives, in the context of capillary-gravity waves,  
sufficient conditions which ensure that $\mu:(\lambda_0,\infty)\to(0,\infty)$ is a bijective mapping, 
which corresponds to the choice $N=1$ in Theorem \ref{T:MT},
situation when no restrictions are needed.
On the other hand, when considering pure capillary waves and if $\mu:(\lambda_0,\infty)\to(0,\infty)$ is a 
bijective mapping, then necessarily $\Gamma_M=\G(p_0),$ and the problems \eqref{ERU} and \eqref{ERUa} become singular as $\lambda\to \lambda_0=2\G_M.$
Therefore,  finding sufficient conditions in this setting appears to be  much more involved.
 \begin{lemma}\label{L:9}
 Let $r\geq3$, $\gamma\in L_r((p_0,0))$ and assume that $g>0$. 
 Then,  $\lambda_0>2\G_M$ and {the integer $N$ in Theorem \ref{T:MT} satisfies $N=1$}, provided that
 \begin{align}\label{eq.condCG}
  \int_{p_0}^0a(\lambda_0)\left(\int_{p_0}^p\frac{1}{a^3(\lambda_0;s)}\, ds\right)^2\, dp<\frac{\sigma}{g^2}.
 \end{align}
\end{lemma}
\begin{proof}
 Let us assume that $\G(p_1)=\G_M$ for some $p_1\in[p_0,0)$ (the case when $p_1=0$ is similar). 
 Then, if $\delta<1$ is such that $p_1+\delta<0,$ we have
 \begin{align*}
  \lim_{\lambda\searrow2\G_M}\int_{p_0}^0 \frac{dp}{a^3(\lambda;p)}&=\lim_{\e\searrow 0}\int_{p_0}^0\frac{dp}{\sqrt{\e+2(\Gamma(p_1)-\Gamma(p))}^3}\geq
  c\lim_{\e\searrow 0}\int_{p_1}^{p_1+\delta}\frac{dp}{\e^{3/2}+\left|\int_{p_1}^p\gamma(s)\, ds\right|^{3/2}}\\
  & \geq c\lim_{\e\searrow 0}\int_{p_1}^{p_1+\delta}\!\frac{dp}{\e^{3/2}+\left\| \gamma \right\|_{L_r}^{3/2}|p-p_1|^{3\alpha/2}}\geq 
  c\lim_{\e\searrow 0}\int_{p_1}^{p_1+\delta}\!\frac{dp}{\e+p-p_1}=\infty
 \end{align*}
with $\alpha=(r-1)/r$ and with $c$ denoting positive constants that are independent of $\e$.
We have used the relation $3\alpha/2\geq1$ for $r\geq3.$
In view of Lemma \ref{L:1}, we  find  that $\lambda_0>2\G_M$ is the unique zero of $W(0;\cdot,0).$  
Recalling now \eqref{qqq} and the relation \eqref{v1}, one can easily see, because of $W(0;\lambda_0,0)=0,$ that the condition  \eqref{eq.condCG} yields  $W_\mu(0;\lambda_0,0)<0$.   
Since Lemma \ref{L:6} implies $W(0;\lambda_0, \inf_{(\lambda_0,\infty)}\mu)=0,$ the   relation $W_\mu(0;\lambda_0,0)<0$ together with Lemma \ref{L:5} guarantee that  $\inf_{(\lambda_0,\infty)}\mu=0$.
This proves the claim.
\end{proof}

\end{document}